\documentclass[a4,10pt,reqno]{amsart}
\usepackage{amsfonts,amssymb}
\theoremstyle{plain}

\newtheorem{theorem}{Theorem}[section]
\newtheorem{proofth1}{Proof of Theorem}[section]

\newtheorem{definition}{Definition}[section]
\newtheorem{lemma}{Lemma}[section]

\newcommand{\FF}{\mathcal{F}}
\newcommand{\E}{\mathbb{E}}

\newcommand{\PP}{\mathbb{P}}

\begin{document}
	\title[Rate of convergence to stable law ]{
Rate of convergence to alpha stable law using Zolotarev distance : Technical report }

\maketitle{}

\centerline{Solym Mawaki MANOU-ABI $^{\rm 1,2}$}

\vspace{12pt}

\centerline{$^{\rm 1}$ CUFR  de Mayotte}
\centerline{D\'epartement Sciences Technologies, Dembeni, France }
\centerline{solym.manou-abi@univ-mayotte.fr }

\vspace{12pt}

\centerline{$^{\rm 2}$ IMAG, Univ Montpellier}
\centerline{ CNRS, Montpellier, France}
\centerline{solym-mawaki.manou-abi@umontpellier.fr}

\renewcommand{\thefootnote}{}
\renewcommand{\thefootnote}{\arabic{footnote}}
\setcounter{footnote}{0}
\date{}
\begin{abstract}
This paper considers the question of the rate of convergence to $\alpha$-stable laws, using arguments based on the Zolotarev distance to prove bounds.  We provide a rate of convergence to 
$\alpha$-stable random variable where $ 1 <\alpha < 2$, in the generalized CLT, that
is, for the partial sums of independent identically distributed  random variables which are not assumed to be square integrable. This work is a technical report based on the Zolotarev paper in \cite{zolotarev}. 
\end{abstract}

\vspace{12pt}

{\bf AMS Subject Classification} : 60F05; 60G52; 60E07.

{\bf Key words and phrases}: CLT, Probability metric, Rate of convergence, Stable law.

\section{Introduction}
\label{intro}
When studying limit theorems in probability theory it is important to try to assess the rates at which these converges. It is known that taking into account probability metrics in the space of random variables, allow to approximate partial sums for sequences of random variables. It is also known that stable laws are the only limits in the study of the limit law for such partial sums properly normalized. However, the lack of explicit formulas for their density functions except some stable distributions (Gaussian, Cauchy, L\'evy) significantly limits their use in practice. A random variable $\vartheta$ has a stable distribution if there exists some coefficients $C_{n}>0$ and $D_{n}\in \mathbb{R}$ such that:
\begin{equation*}
  \vartheta_{1}+\ldots +\vartheta_{n} \overset{d}{=} C_{n}\vartheta+D_{n}  \quad \forall n\geq 1,
  \end{equation*}
 where  $\vartheta_{1},\ldots, \vartheta_{n}$ are i.i.d copies of $\vartheta$.
In particular such random variables are infinitely divisible. It is shown that $C_{n}=\,n^{1/\alpha}$ with $\alpha \in (0,2]$ (see for instance \cite{feller}, Section VI.1). The coefficient
 $\alpha$ is the index of  stability. We shall say that  $\vartheta$ is an $\alpha$-stable (real-valued) random variable. The case $\alpha=2$  and $D_{n}=0$ corresponds to a Gaussian random variable. The L\'evy-Khintchine representation in the case  $\alpha \in (0,2)$ reads as follows:
 $  \E [e^{iu \vartheta}] = e^{\psi(u)}$  where $\psi$ is the characteristic exponent given by
\[\psi(u)=ibu+ \int_{-\infty}^{+\infty} ( e^{iu x}-1-iu x1_{|x|\leq 1})\frac{1}{|x|^{1+\alpha}}(c_{1} 1_{x<0}+ c_{2} 1_{x>0} )dx,\]
with $b\in \mathbb{R} $, $c_{1},c_{2} \geq 0$ and $c_{1}+c_{2}>0$. We say in this case that $\vartheta$ is an $\alpha$-stable random variable with characteristics $(b,c_{1},c_{2})$.  For strictly $\alpha$-stable random variable we have $b=0$ and the case $b=0$, $c_{1}=c_{2}$ corresponds to a symmetric $\alpha$-stable random variable.\\
In the context of the Central Limit Theorem (CLT) it is well known that, for some given sequence of iid random variables $(V_{n})_{n\geq 1}$, the partial sum properly normalized convergence to an $\alpha$-stable random variable, $\alpha \in (0,2]$, if the random variable $V_{1}$ is in a domain of attraction of the $\alpha$-stable random variable. We say that
the random variable $V_{1}$ is in the domain of attraction of an $\alpha$-stable with tails parameters $c_{1}$ and $c_{2}$ if
\begin{equation}
\label{eq:stable}
  \PP[V_{1} >x] = \frac{c_{1} +h(x)}{x^{\alpha}}, \quad \textrm{and} \quad
  \PP[ V_{1} \leq -x] = \frac{c_{2} +h(x)}{x^{\alpha}},  \quad x>0,
  \end{equation}
where $h$ is a function such that $\lim_{x\rightarrow \infty} h(x)=0$. The case of domain of attraction of a symmetric
$\alpha$-stable random variable corresponds to $c_{1}=c_{2}$.
The reader can refer to \cite{gnedenko1968} for the domain of attraction conditions.

We shall also say that the random variable $X$ is in a strong domain of attraction of the $\alpha$-stable if in (\ref{eq:stable}), there exists $\gamma_{\alpha}>0$ such that $h(x) = \mathcal{O}(\frac{1}{|x|^{\gamma_{\alpha}}})$
 at $+\infty$ meaning that $\exists K, x_{0} \in \mathbb{R}_{+}$ and $|h(x)|\leq K |x|^{-\gamma_{\alpha}}$ for all $x\geq x_{0}$ :
 $$  \PP[V_{1}  >x] - \frac{c_{1}}{x^{\alpha}}=\mathcal{O}(\frac{1}{x^{\alpha+\gamma_{\alpha}}}),
  \quad \PP[V_{1}  \leq -x] - \frac{c_{2}}{|x|^{\alpha}} = \mathcal{O}(\frac{1}{|x|^{\alpha+\gamma_{\alpha}}}),  \quad x>0 \, . $$
In the sequel we call $\gamma_{\alpha}$ the attraction index.\\

In this paper, we obtain a rate of convergence to an $\alpha$-stable random variable, $\alpha \in (1,2)$, for partial sums of iid  random variables $(V_{n})_{n\geq 1}$ (the generalized CLT) between characteristic functions.

To this end, we show in Theorem \ref{theo:independance} a bound  of order $n^{\frac{\alpha-r}{\alpha}}$ in the Zolotarev ideal probability metric $\zeta_{r}$ for $r\in ]\alpha,2]$, provided the random variable $V_{1}$ is in a strong domain of attraction of the $\alpha$-stable random variable with an attraction index $\gamma_{\alpha} >  r-\alpha$. We then applied the result to derive the rate of convergence  between characteristic functions, see Theorem \ref{maintheo}. 

The layout of this short paper is as follows. We start in Section \ref{preliminaires}, by the notion of ideal probability metrics and introduce  the Zolotarev ideal probability metric of order $r>0$. We give in Section \ref{main} a rate of convergence  to an $\alpha$-stable random variable, $\alpha \in (1,2)$, for iid random variables  between characteristic functions in  Theorem \ref{maintheo}. 
\section{Preliminaries}
\label{preliminaires}
Denote by $\Sigma $ the space of real random variables. In this section we introduce the notion of ideal probability distance and as an example we introduce the Zolotarev distance.\\
On a probability space $(\Omega,\FF,\PP)$ consider two real random variables $X$ and $Y$, denote by $P_{X}$ and $ P_Y$ their probability laws respectively.
\begin{definition}
 A map $d(.,.)$ defined in the space $\Sigma \times \Sigma \rightarrow [0,\infty]$ is said to
 be a probability distance in  $\Sigma$ if for all random variables  $X$, $Y$ and $Z$ the following statements hold:
\begin{itemize}
\item[(1)] \[\mathbb{P}(X=Y)=1 \Rightarrow d(X,Y)=0, \]
\item[(2)] \[ d(X,Y)= d(Y,X), \]
\item[(3)]  \[ d(X,Y) \leq d(X,Z)+ d(Z,Y). \]
\end{itemize}
\end{definition}

If the values of $d(X,Y)$ are determined by the marginal distributions $P_X$ and $P_Y$ then one says that the distance $d$ is simple.

An example of a  simple distance is:
\begin{itemize}
 \item  the total variation distance defined by:
\[ d_{VT}(X,Y)= \sup_{||f||_{\infty}\leq 1} \vert \E [f(X)]- \E[f(Y)] \vert, \]
where $||f||_{\infty}=\sup_{x\in \mathbb{R}} |f(x)|$.
\end{itemize}
Now let us define an ideal probability distance. 
\begin{definition}
A simple distance $d$ in  $\Sigma$ is called an ideal probability distance of order $r\geq 0$, if the following statements hold:
\begin{itemize}
\item[(4)]
\begin{equation}
\label{eq:regularity}
 d(X+Z, Y+Z) \leq d(X,Y),
 \end{equation}
for $Z$ independent of $X$ and $Y$. (Regularity)
\item[(5)]
\begin{equation}
\label{eq:homogeneity}
 d(cX,cY)=|c|^r d(X,Y)
 \end{equation}
 for any $c \in \mathbb{R}^{*}$. (Homogeneity of order $r$)
\end{itemize}
\end{definition}
We now give as  example  the Zolotarev distance. 
\begin{definition}
Let $r>0$ with the representation $r=m+\beta$ where $\beta \in ]0,1]$ and $m \in \mathbb{N}$.
We define the following simple metric :
 \begin{equation}
\zeta_{r}(X,Y)= \sup_{f\in \Lambda_r} \vert \E [f(X)] - \E [f(Y)] \vert,
\end{equation}
where $\Lambda_r$ is the set of bounded functions $f: \mathbb{R} \rightarrow \mathbb{R}$ which are $m$-times continuously differentiable and such that
\begin{equation}
 |f^{(m)}(x)- f^{(m)}(y)| \leq |x-y|^{\beta}, \quad \beta+m=r,
 \end{equation}
where  $f^{(m)}$ is the derivate function of order $m$.\\
\end{definition}
The metric $\zeta_{r}$ is called the Zolotarev probability distance and it is an ideal probability metric of order $r$.
We refer the reader to Theorem 1.4.2 in \cite{zolotarev2} for the proof.\\
Note that by a simple application of the Taylor formula with integral remainder, one can show that
$\zeta_{r}(X,Y)<\infty$ provided
  \begin{equation}
    \label{eq:moments}
    \E[|X|^{r}], \, \E[|Y|^{r}] < \infty
    \end{equation}
  and
 \begin{equation}
 \label{eq:mixmoment}
  \E [X^{k}] = \E [Y^{k}] \quad 0\leq k \leq m.
 \end{equation}
Note also that the condition in (\ref{eq:mixmoment}) is necessarily but the moment condition in (\ref{eq:moments}) can be relaxed using domain of attraction conditions. 

Some interesting cases are $r=1$ and $r=2$.
\begin{itemize}
\item[(a)] The Zolotarev metric of order $1$ is defined by:
\[ \zeta_{1}(X,Y) =\sup_{f\in \Lambda_1} \left| \E [f(X)] - \E [f(Y)] \right|,  \]
 where $\Lambda_{1}$ is the set of bounded continuous functions $f : \mathbb{R} \rightarrow \mathbb{R}$
 such that  $ |f(x)-f(y)| \leq |x-y|$. \\
  By the famous Kantorovich-Rubinstein duality, it rewrites as :
 \[ \mathbb{W}(X,Y)=\inf_{(P_{X},P_{Y})} \E[\left|X-Y\right|],  \]
where the infimum runs over all coupling of the marginal distributions $P_{X}$ and $P_{Y}$. 
\item[(b)] The Zolotarev metric of order $2$ is defined by 
\[ \zeta_{2}(X,Y) =\sup_{f \in \Lambda_2} \vert \E [f(X)] - \E [f(Y)] \vert,  \]
where $\Lambda_{2}$ is the set of bounded functions $f: \mathbb{R} \rightarrow \mathbb{R}$ which are $1$-times continuously differentiable and such that $|f'(x)-f'(y)| \leq |x-y|$.
\end{itemize}

\section[Rates of convergence] {Main results}
\label{main}
\begin{lemma}
\[
|x-y|\max(|x|^{r-1},|y|^{r-1}) \leq 2\left| x|x|^{r-1}-y|y|^{r-1} \right|,\quad \forall x,y \in \mathbb{R},\quad r\geq 1. \]
\end{lemma}
\begin{proof}
If we apply the following triangle inequality
\[ \left| |a|-|b| \right| \leq |a-b|, \]
for  $a=x|x|^{r-1}$  and  $b=x|x|^{r-1}$, then we have:
\begin{equation*}
\left| |x|^{r}-|y|^{r} \right| \leq \left| x|x|^{r-1}-y|y|^{r-1} \right|,
\end{equation*}
so that
\begin{equation}
\label{eq:inegalite1}
 \left| x|x|^{r-1}-y|y|^{r-1} \right| + \left| |x|^{r}-|y|^{r} \right| \leq 2\left| x|x|^{r-1}-y|y|^{r-1} \right|.
\end{equation}
Now  assume that $|x|\geq |y|$. We have
 \begin{eqnarray*}
      |x-y||x|^{r-1} &=& \vert  x|x|^{r-1}-y|x|^{r-1}  \vert \\
      &=& \vert x|x|^{r-1}-y|y|^{r-1}+y|y|^{r-1}-y|x|^{r-1} \vert \\
      &\leq & \left| x|x|^{r-1}-y|y|^{r-1} \right| + |y|(|x|^{r-1}-|y|^{r-1}) \\
      &=&  \left| x|x|^{r-1}-y|y|^{r-1} \right| + |y| |x|^{r-1}-|y|^{r}\\
      &\leq & \left| x|x|^{r-1}-y|y|^{r-1} \right| + |x|^{r}-|y|^{r}.
      \end{eqnarray*}
This leads to the following fact:
\begin{equation}
\label{eq:inegalite2}
 |x-y|\max(|x|^{r-1},|y|^{r-1}) \leq \left| x|x|^{r-1}-y|y|^{r-1} \right| + \left| |x|^{r}-|y|^{r} \right|.
 \end{equation}
Bringing together the inequalities (\ref{eq:inegalite1}) and (\ref{eq:inegalite2}), we thus obtain
the desired inequality.\\
\end{proof}

\label{zolotarev}
\begin{theorem}
\label{theo:independance}
Given a sequence of integrable iid random variables $(V_n)_{n\geq 1}$, set
$$\tilde{S}_{n}:= n^{-1/\alpha}\sum_{k=1}^{n} (V_{k}-\E [V_{k}]).       $$
Assume that $V_{1}$ is in the strong domain of attraction of a symmetric $\alpha$-stable random variable $\vartheta$ with $\alpha \in (1,2)$ and an attraction index $\gamma_{\alpha}>r-\alpha$ where $r\in (\alpha,2]$.  Then there exists a constant $C>0$ such that :
\[ \zeta_{r}(\tilde{S}_{n}, \vartheta ) \leq C\, n^{\frac{\alpha-r}{\alpha}}.\]
\end{theorem}
One can deduce the rate of convergence between characteristic functions.
\begin{theorem}
\label{maintheo}
Given a sequence of integrable iid random variables $(V_n)_{n\geq 1}$, set
$$\tilde{S}_{n}:= n^{-1/\alpha}\sum_{k=1}^{n} (V_{k}-\E [V_{k}]).       $$
Assume that $V_{1}$ is in the strong domain of attraction of a symmetric $\alpha$-stable random variable $\vartheta$ with $\alpha \in (1,2)$ and an attraction index $\gamma_{\alpha}>2-\alpha$.  Then there exists a constant $C(t)>0$
 such that :
\[  \chi_{t}(\tilde{S}_{n}, \vartheta ) = \vert \E e^{it \tilde{S}_{n}} - \E e^{it \vartheta} \vert  \leq C(t)\, n^{\frac{\alpha-2}{\alpha}}.
\]
\end{theorem}

\begin{proof}
Set
 \[ \chi_{t}(\tilde{S}_{n}, \vartheta )= \left| \E [e^{it\tilde{S}_{n}}]- \E [e^{it\vartheta}] \right|, \quad t\in \mathbb{R},\]
and observe that, using the definition of the ideal probability distance $\zeta_{2}$, we have
 \[ \chi_{t}(\tilde{S}_{n}, \vartheta ) \leq t^{2} \zeta_{2}(\tilde{S}_{n}, \vartheta ) ,  \]
since the function $f_{t}(x)=e^{itx}, \; t\in \mathbb{R}, \; x\in \mathbb{R}$ is bounded and the derivative $f'_{t}$ is a $t^{2}$-Lipschitz function. \\
We thus obtain the result by applying Theorem 3.1 with $r=2$. 

\end{proof}

Now let us start the proof of  Theorem 3.1

\begin{proofth1}
Without loss of generality, we assume that the sequence $(V_n)_{n\geq 1}$ is centered.
Consider a sequence $(\vartheta_n)_{n\geq 1}$ of i.i.d copies of $\vartheta$. We have the following identity:
\begin{equation*}
\vartheta \overset{\mathcal{L}}{=} \frac{\vartheta_1 + \ldots+ \vartheta_n }{n^{1/\alpha}}, \quad \forall n\geq 1.
\end{equation*}
We choose the sequence $(\vartheta_n)_{n\geq 1}$ to be independent of the sequence $(V_{n})_{n\geq 1}$.
Since $\zeta_{r}$ is a simple distance, we have
\[  \zeta_{r}(n^{-1/\alpha}\sum_{k=1}^{n} V_{k}, \vartheta)= \zeta_{r}\left(n^{-1/\alpha}\sum_{k=1}^{n} V_{k}, n^{-1/\alpha}\sum_{k=1}^{n}\vartheta_{k} \right).      \]
For $n=2$, since $\zeta_{r}$ is an ideal probability metric of order $r$ we have:

\begin{eqnarray*}
\zeta_{r} \left(2^{-1/\alpha}(V_{1}+ V_{2}), 2^{-1/\alpha}(\vartheta_{1}+\vartheta_{2})\right) &\leq &
\zeta_{r}\left(2^{-1/\alpha}(V_{1}+  V_{2}), 2^{-1/\alpha}(\vartheta_{2}+ V_{1})\right)\\
 && + \zeta_{r}\left( 2^{-1/\alpha}(\vartheta_{2}+ V_{1}), 2^{-1/\alpha}(\vartheta_{1}+  \vartheta_{2})\right)\\
  &\leq & 2^{-\frac{r}{\alpha}}\left(\zeta_{r}(V_{1}, \vartheta_{1})+ \zeta_{r}(V_{2}, \vartheta_{2}) \right)\\
  &=& 2^{1-\frac{r}{\alpha}}\, \zeta_{r}(V_{1}, \vartheta_{1})= 2^{1-\frac{r}{\alpha}}\,\zeta_{r}(V_{1}, \vartheta),
\end{eqnarray*}
where we use in the first inequality the regularity property and in the second inequality the homogeneity of order $r$ of the probability metric $\zeta_{r}$; the last equality holds since $\zeta_{r}$ is a simple distance.\\
Thus, by induction on $n$, we have
\begin{eqnarray*}
\label{eq:star1}
\zeta_{r}\left(n^{-1/\alpha}\sum_{k=1}^{n} V_k, n^{-1/\alpha}\sum_{k=1}^{n}\vartheta_{k}\right) &\leq & n^{-\frac{r}{\alpha}} \sum_{k=1}^{n} \zeta_{r}(V_{k}, \vartheta_{k}),
\end{eqnarray*}
so that
\begin{eqnarray*}
\zeta_{r}\left(n^{-1/\alpha}\sum_{k=1}^{n} V_{k}, n^{-1/\alpha}\sum_{k=1}^{n}\vartheta_{k}\right)
&\leq & n^{1-\frac{r}{\alpha}}\zeta_{r}(V_{1}, \vartheta_{1})\\
&= &  n^{1-\frac{r}{\alpha}} \zeta_{r}(V_{1},\vartheta)\\
&=& C\,n^{1-\frac{r}{\alpha}},
\end{eqnarray*}
provided $C:= \zeta_{r}(V_{1}, \vartheta)$ is finite.\\
Now the remainder of the proof is devoted to show that $C$ is finite. \\

Consider two real random variables $X$, $Y$ and define 
\[ \kappa_{r}(X,Y)=\sup_{f\in \Gamma} \vert \E[f(X)]-\E[f(Y)] \vert, \]
where $\Gamma$ is the set of bounded functions $f: \mathbb{R} \rightarrow \mathbb{R}$ such that 
\[ |f(x)-f(y)| \leq \left|x|x|^{r-1}-y|y|^{r-1}\right|.   \]
By  \cite{rachev1990}, $\kappa_{r}$ rewrites as 
    \[  \kappa_{r}(X,Y)=  \inf_{(P_{X},P_{Y})}
 \left( \E\left[\left| X|X|^{r-1}-Y|Y|^{r-1}\right| \right] \right), \]
where the infimum runs over all coupling of the  marginal distributions $P_{X}$ and $P_{Y}$.\\

Let $V_1$ and $\vartheta$ be the optimal coupling for $ \kappa_{r}(V_{1},\vartheta)$ and $f: \mathbb{R} \rightarrow \mathbb{R}$, bounded and $1$-time continuously differentiable such that: 
\[ |f'(x)- f'(y)| \leq |x-y|^{r-1}. \]
Letting $Z=V_{1}-\vartheta$, we have
\[ f(V_{1})-f(\vartheta)= f(\vartheta+Z)-f(\vartheta).  \]
By the mean value theorem there exists  $ \lambda \in ]0,1]$ such that:
\[   f(\vartheta+Z)-f(\vartheta)  = f'(\vartheta+\lambda Z)Z .          \]
Since $\E [V_{1}] = \E [\vartheta] = 0$ we have,
\[ \E [f(\vartheta+Z)-f(\vartheta)]  = \E [f'(\vartheta +\lambda Z)Z] - \E[f'(0)Z] \]
and
 \begin{eqnarray*}
  \vert \E [f(\vartheta+Z)]- \E[f(\vartheta)] \vert   &\leq&  \E \left[\vert f'(\vartheta+\lambda Z)Z - f'(0)Z \vert\right]  \\
   &\leq &  \E \left[ \vert \vartheta +\lambda Z\vert^{r-1} |Z|\right]\\
  &= & \E \left[ \vert (1-\lambda) \vartheta +\lambda V_{1} \vert^{r-1} |V_{1}-\vartheta|\right].
 \end{eqnarray*}
  Since 
 $$ (1-\lambda)|x| + \lambda|y| \leq \max(|x|,|y|), \quad \lambda \in ]0,1],          $$ 
 we have,
   \begin{equation*}
   \label{eq:maxbis} 
  \vert \E [f(\vartheta+Z)]- \E[f(\vartheta)] \vert  \leq \E\left[ |V_{1}-\vartheta|\max(|V_{1}|^{r-1},|\vartheta|^{r-1})\right]. 
  \end{equation*}
Now taking into account Lemma 3.1 we obtain
\[ \vert \E [f(\vartheta+Z)]- \E[f(\vartheta)] \vert \leq 2\,\E\left[\left|V_{1}|V_{1}|^{r-1}-\vartheta|\vartheta|^{r-1}\right|\right], \]
and thus 
 \begin{eqnarray*}
\zeta_{r}(V_{1},\vartheta) &\leq & 2\, \E\left[ \left|V_{1}|V_{1}|^{r-1}-\vartheta|\vartheta|
^{r-1}\right| \right] \\
& = & 2 \, \kappa_{r}(V_{1},\vartheta).
 \end{eqnarray*}
By \cite{zolotarev2}, the following representation hold:
\[  \kappa_{r}(V_{1},\vartheta) \, = \,  r\, \int_{\mathbb{R}} |u|^{r-1} |F_{V_{1}}(u)-F_{\vartheta}(u)|du, \]
and thus 
 \begin{eqnarray*}
 \kappa_{r}(V_{1},\vartheta)&=& r \int_{0}^{+\infty} |u|^{r-1} \left| \PP[V_{1}>u]-\PP[\vartheta>u] \right| du\\
 &+& r \int_{0}^{+\infty}|u|^{r-1} \left| \PP[V_{1}<-u]-\PP[\vartheta<-u] \right| du.
\end{eqnarray*}
Now recall that for any $\alpha$-stable random variable $\Theta$ with local characteristic
$(c_{1},c_{2})$, we have the following expansion, cf. \cite{mijn}, when $u \rightarrow +\infty$:
\begin{equation*}
\label{eq:asymptotic}
\PP[\Theta>u]\,= \, \frac{c_1}{u^{\alpha}} + \frac{c_2}{u^{2\alpha}} + \mathcal{O}(\frac{1}{u^{3\alpha}}), \quad
  \PP[\Theta<-u]\,= \, \frac{c_2}{u^{\alpha}} + \frac{c_1}{u^{2\alpha}} + \mathcal{O}(\frac{1}{u^{3\alpha}}).
\end{equation*}
This means that every stable random variable is in its own strong domain of attraction : 
\[  \left| \PP[V_{1}>u]-\PP[\vartheta>u]\right| \; \leq \; \left|\PP[V_{1}>u]-c\,u^{-\alpha}\right|  + \mathcal{O}(u^{-2\alpha}), \]
and \[ \left| \PP[V_{1}<-u]-\PP[\vartheta<-u] \right|\; \leq \; \left|\PP[V_{1}<-u]-c\,u^{-\alpha}\right| + \mathcal{O}(u^{-2\alpha}),\]
where $c$ is the local characteristic of the symmetric $\alpha$-stable random variable $\vartheta$ with $\alpha\in (1,2)$.\\
 Therefore we get \[\kappa_{r}(V_{1},\vartheta) < \infty \quad  \Longleftrightarrow \quad
\int_{1}^{\infty} |u|^{r-1-\alpha-\gamma_{\alpha}}\,du < \infty \quad \Longleftrightarrow \quad
\gamma_{\alpha}>r-\alpha,\]
which corresponds to our assumptions. Since 
\[ \zeta_{r}(V_{1},\vartheta) \,\leq\, 2\, \kappa_{r}(V_{1},\vartheta),  \]
we then have : \[ \zeta_{r}(\tilde{S}_{n}, \vartheta ) \leq C\, n^{\frac{\alpha-r}{\alpha}}.\]

\end{proofth1}

\end{document}